\documentclass{amsart}
\usepackage{amssymb}
\usepackage{amsfonts}

\newtheorem{theorem}{Theorem}
\theoremstyle{plain}

\newtheorem{corollary}{Corollary}

\newtheorem{lemma}{Lemma}

\newtheorem{proposition}{Proposition}
\newtheorem{remark}{Remark}

\numberwithin{equation}{section}
\input{tcilatex}

\begin{document}
\title[HERMITE-HADAMARD TYPE INEQUALITIES]{NEW INEQUALITIES OF
HERMITE-HADAMARD TYPE FOR CONVEX FUNCTIONS WITH APPLICATIONS}
\author{HAVVA KAVURMACI$^{\clubsuit ,\blacktriangle }$}
\address{$^{\blacktriangle }$ATAT\"{U}RK UNIVERSITY, K.K. EDUCATION FACULTY,
DEPARTMENT OF MATHEMATICS, 25240, CAMPUS, ERZURUM, TURKEY}
\email{havva.kvrmc@yahoo.com}
\thanks{$^{\clubsuit }$Corresponding Author}
\author{MERVE AVCI$^{\blacktriangle }$}
\email{merveavci@ymail.com}
\author{M. Emin \"{O}ZDEM\.{I}R$^{\blacktriangle }$}
\email{emos@atauni.edu.tr}

\begin{abstract}
In this paper, some new inequalities of the Hermite-Hadamard type for
functions whose modulus of the derivatives are convex and applications for
special means are given. Finally, some error estimates for the trapezoidal
formula are obtained.
\end{abstract}

\maketitle

\section{INTRODUCTION}

A function $f:I\rightarrow 
\mathbb{R}
$ is said to be convex function on $I$ if the inequality

\begin{equation*}
f(\alpha x+(1-\alpha )y)\leq \alpha f(x)+(1-\alpha )f(y),
\end{equation*}%
holds for all $x,y\in I$ and $\alpha \in \left[ 0,1\right] $.

One of the most famous inequality for convex functions is so called
Hermite-Hadamard's inequality as follows: Let $f:I\subseteq 
\mathbb{R}
\rightarrow 
\mathbb{R}
$ be a convex function defined on the interval $I$ of real numbers and $%
a,b\in I$, with $a<b$. Then :

\begin{equation}
f\left( \frac{a+b}{2}\right) \leq \frac{1}{b-a}\int_{a}^{b}f(x)dx\leq \frac{%
f(a)+f(b)}{2}.  \label{1.1}
\end{equation}

In \cite{SDRA}, the following theorem which was obtained by Dragomir and
Agarwal contains the Hermite-Hadamard type integral inequality.

\begin{theorem}
\label{t.1.1} Let $f:I^{\circ }\subseteq 
\mathbb{R}
\rightarrow 
\mathbb{R}
$ be a differentiable mapping on $I^{\circ }$, $a,b\in I^{\circ }$ with $%
a<b. $ If $\left\vert f^{\prime }\right\vert $ is convex on $\left[ a,b%
\right] $, then the following inequality holds:%
\begin{equation}
\left\vert \frac{f(a)+f(b)}{2}-\frac{1}{b-a}\int_{a}^{b}f(u)du\right\vert
\leq \frac{(b-a)\left( \left\vert f^{\prime }(a)\right\vert +\left\vert
f^{\prime }(b)\right\vert \right) }{8}.  \label{1.2}
\end{equation}
\end{theorem}

In \cite{UMMJ} K\i rmac\i , Bakula, \"{O}zdemir and Pe\v{c}ari\'{c} proved
the following theorem.

\begin{theorem}
\label{t.1.2} Let $f:I\rightarrow 
\mathbb{R}
$, $I\subset 
\mathbb{R}
$, be a differentiable function on $I^{\circ }$ such that $f^{\prime }\in L%
\left[ a,b\right] $, where $a,b\in I$, $a<b$. If $\left\vert f^{\prime
}\right\vert ^{q}$ is concave on $\left[ a,b\right] $ for some $q>1$, then:%
\begin{eqnarray}
&&\left\vert \frac{f(a)+f(b)}{2}-\frac{1}{b-a}\int_{a}^{b}f(u)du\right\vert
\label{1.3} \\
&\leq &\left( \frac{b-a}{4}\right) \left[ \frac{q-1}{2q-1}\right] ^{\frac{q-1%
}{q}}\left( \left\vert f^{\prime }\left( \frac{a+3b}{4}\right) \right\vert
+\left\vert f^{\prime }\left( \frac{3a+b}{4}\right) \right\vert \right) . 
\notag
\end{eqnarray}
\end{theorem}

For recent results and generalizations concerning Hermite-Hadamard's
inequality see \cite{CPJP}-\cite{UMMJ} and the references therein.

\section{THE NEW HERMITE-HADAMARD TYPE INEQUALITIES}

In order to prove our main theorems, we first prove the following lemma:

\begin{lemma}
\label{2.1} Let $f:I\subseteq 
\mathbb{R}
\rightarrow 
\mathbb{R}
$ be a differentiable mapping on $I^{\circ }$(the interior of $I$), where $%
a,b\in I$ with $a<b$. If $f^{\prime }\in L\left[ a,b\right] $, then the
following inequality holds:%
\begin{eqnarray*}
&&\frac{\left( b-x\right) f(b)+(x-a)f(a)}{b-a}-\frac{1}{b-a}%
\int_{a}^{b}f(u)du \\
&=&\frac{\left( x-a\right) ^{2}}{b-a}\int_{0}^{1}\left( t-1\right) f^{\prime
}(tx+(1-t)a)dt+\frac{\left( b-x\right) ^{2}}{b-a}\int_{0}^{1}\left(
1-t\right) f^{\prime }(tx+(1-t)b)dt.
\end{eqnarray*}

\begin{proof}
We note that%
\begin{eqnarray*}
I &=&\frac{\left( x-a\right) ^{2}}{b-a}\int_{0}^{1}\left( t-1\right)
f^{\prime }(tx+(1-t)a)dt \\
&&+\frac{\left( b-x\right) ^{2}}{b-a}\int_{0}^{1}\left( 1-t\right) f^{\prime
}(tx+(1-t)b)dt.
\end{eqnarray*}%
Integrating by parts, we get%
\begin{eqnarray*}
I &=&\frac{\left( x-a\right) ^{2}}{b-a}\left[ \left. \left( t-1\right) \frac{%
f(tx+(1-t)a)}{x-a}\right\vert _{0}^{1}-\int_{0}^{1}\frac{f(tx+(1-t)a)}{x-a}dt%
\right] \\
&&+\frac{\left( b-x\right) ^{2}}{b-a}\left[ \left. \left( 1-t\right) \frac{%
f(tx+(1-t)b)}{x-b}\right\vert _{0}^{1}+\int_{0}^{1}\frac{f(tx+(1-t)b)}{x-b}dt%
\right] \\
&=&\frac{\left( x-a\right) ^{2}}{b-a}\left[ \frac{f(a)}{x-a}-\frac{1}{\left(
x-a\right) ^{2}}\int_{a}^{x}f(u)du\right] \\
&&+\frac{\left( b-x\right) ^{2}}{b-a}\left[ -\frac{f(b)}{x-b}+\frac{1}{%
\left( x-b\right) ^{2}}\int_{b}^{x}f(u)du\right] \\
&=&\frac{(b-x)f(b)+(x-a)f(a)}{b-a}-\frac{1}{b-a}\int_{a}^{b}f(u)du\text{.}
\end{eqnarray*}
\end{proof}
\end{lemma}

Using the Lemma \ref{2.1} the following result can be obtained.

\begin{theorem}
\label{t.2.1} Let $f:I\subseteq 
\mathbb{R}
\rightarrow 
\mathbb{R}
$ be a differentiable mapping on $I^{\circ }$such that $f^{\prime }\in L%
\left[ a,b\right] $, where $a,b\in I$ with $a<b$. If $\left\vert f^{\prime
}\right\vert $ is convex on $\left[ a,b\right] $, then the following
inequality holds:%
\begin{eqnarray*}
&&\left\vert \frac{(b-x)f(b)+(x-a)f(a)}{b-a}-\frac{1}{b-a}%
\int_{a}^{b}f(u)du\right\vert \\
&\leq &\frac{\left( x-a\right) ^{2}}{b-a}\left[ \frac{\left\vert f^{\prime
}(x)\right\vert +2\left\vert f^{\prime }(a)\right\vert }{6}\right] +\frac{%
\left( b-x\right) ^{2}}{b-a}\left[ \frac{\left\vert f^{\prime
}(x)\right\vert +2\left\vert f^{\prime }(b)\right\vert }{6}\right]
\end{eqnarray*}%
for each $x\in \left[ a,b\right] .$
\end{theorem}

\begin{proof}
Using Lemma \ref{2.1} and taking the modulus, we have%
\begin{eqnarray*}
&&\left\vert \frac{(b-x)f(b)+(x-a)f(a)}{b-a}-\frac{1}{b-a}%
\int_{a}^{b}f(u)du\right\vert \\
&\leq &\frac{\left( x-a\right) ^{2}}{b-a}\int_{0}^{1}\left( 1-t\right)
\left\vert f^{\prime }(tx+(1-t)a)\right\vert dt \\
&&+\frac{\left( b-x\right) ^{2}}{b-a}\int_{0}^{1}\left( 1-t\right)
\left\vert f^{\prime }(tx+(1-t)b)\right\vert dt.
\end{eqnarray*}%
Since $\left\vert f^{\prime }\right\vert $ is convex, then we get%
\begin{eqnarray*}
&&\left\vert \frac{(b-x)f(b)+(x-a)f(a)}{b-a}-\frac{1}{b-a}%
\int_{a}^{b}f(u)du\right\vert \\
&\leq &\frac{\left( x-a\right) ^{2}}{b-a}\int_{0}^{1}\left( 1-t\right) \left[
t\left\vert f^{\prime }(x)\right\vert +(1-t)\left\vert f^{\prime
}(a)\right\vert \right] dt \\
&&+\frac{\left( b-x\right) ^{2}}{b-a}\int_{0}^{1}\left( 1-t\right) \left[
t\left\vert f^{\prime }(x)\right\vert +(1-t)\left\vert f^{\prime
}(b)\right\vert \right] dt \\
&=&\frac{\left( x-a\right) ^{2}}{b-a}\left[ \frac{\left\vert f^{\prime
}(x)\right\vert +2\left\vert f^{\prime }(a)\right\vert }{6}\right] +\frac{%
\left( b-x\right) ^{2}}{b-a}\left[ \frac{\left\vert f^{\prime
}(x)\right\vert +2\left\vert f^{\prime }(b)\right\vert }{6}\right]
\end{eqnarray*}%
which completes the proof.
\end{proof}

\begin{corollary}
\label{c.2.1} In Theorem \ref{t.2.1}, if we choose $x=\frac{a+b}{2}$, we
obtain%
\begin{equation*}
\left\vert \frac{f(a)+f(b)}{2}-\frac{1}{b-a}\int_{a}^{b}f(u)du\right\vert
\leq \frac{b-a}{12}\left( \left\vert f^{\prime }(a)\right\vert +\left\vert
f^{\prime }\left( \frac{a+b}{2}\right) \right\vert +\left\vert f^{\prime
}(b)\right\vert \right) .
\end{equation*}
\end{corollary}

\begin{remark}
\label{r.2.1} In Corollary \ref{c.2.1}, using the convexity of $\left\vert
f^{\prime }\right\vert $, we have%
\begin{equation*}
\left\vert \frac{f(a)+f(b)}{2}-\frac{1}{b-a}\int_{a}^{b}f(u)du\right\vert
\leq \frac{b-a}{8}\left( \left\vert f^{\prime }(a)\right\vert +\left\vert
f^{\prime }(b)\right\vert \right)
\end{equation*}%
which is the inequality in (\ref{1.2}).
\end{remark}

\begin{theorem}
\label{t.2.2} Let $f:I\subseteq 
\mathbb{R}
\rightarrow 
\mathbb{R}
$ be a differentiable mapping on $I^{\circ }$such that $f^{\prime }\in L%
\left[ a,b\right] $, where $a,b\in I$ with $a<b$. If $\left\vert f^{\prime
}\right\vert ^{\frac{p}{p-1}}$ is convex on $\left[ a,b\right] $ and for
some fixed $p>1$, then the following inequality holds:%
\begin{eqnarray*}
&&\left\vert \frac{\left( b-x\right) f(b)+(x-a)f(a)}{b-a}-\frac{1}{b-a}%
\int_{a}^{b}f(u)du\right\vert \\
&\leq &\left( \frac{1}{p+1}\right) ^{\frac{1}{p}}\left( \frac{1}{2}\right) ^{%
\frac{1}{q}} \\
&&\times \left[ \frac{\left( x-a\right) ^{2}\left[ \left\vert f^{\prime
}(a)\right\vert ^{q}+\left\vert f^{\prime }(x)\right\vert ^{q}\right] ^{%
\frac{1}{q}}+\left( b-x\right) ^{2}\left[ \left\vert f^{\prime
}(x)\right\vert ^{q}+\left\vert f^{\prime }(b)\right\vert ^{q}\right] ^{%
\frac{1}{q}}}{b-a}\right]
\end{eqnarray*}%
for each $x\in \left[ a,b\right] $ and $q=\frac{p}{p-1}$.
\end{theorem}

\begin{proof}
From Lemma \ref{2.1} and using the well-known H\"{o}lder integral
inequality, we have 
\begin{eqnarray*}
&&\left\vert \frac{\left( b-x\right) f(b)+(x-a)f(a)}{b-a}-\frac{1}{b-a}%
\int_{a}^{b}f(u)du\right\vert \\
&\leq &\frac{\left( x-a\right) ^{2}}{b-a}\int_{0}^{1}\left( 1-t\right)
\left\vert f^{\prime }(tx+(1-t)a)\right\vert dt \\
&&+\frac{\left( b-x\right) ^{2}}{b-a}\int_{0}^{1}\left( 1-t\right)
\left\vert f^{\prime }(tx+(1-t)b)\right\vert dt \\
&\leq &\frac{\left( x-a\right) ^{2}}{b-a}\left( \int_{0}^{1}\left(
1-t\right) ^{p}dt\right) ^{\frac{1}{p}}\left( \int_{0}^{1}\left\vert
f^{\prime }(tx+(1-t)a)\right\vert ^{q}dt\right) ^{\frac{1}{q}} \\
&&+\frac{\left( b-x\right) ^{2}}{b-a}\left( \int_{0}^{1}\left( 1-t\right)
^{p}dt\right) ^{\frac{1}{p}}\left( \int_{0}^{1}\left\vert f^{\prime
}(tx+(1-t)b)\right\vert ^{q}dt\right) ^{\frac{1}{q}}.
\end{eqnarray*}%
Since $\left\vert f^{\prime }\right\vert ^{\frac{p}{p-1}}$ is convex, by the
Hermite-Hadamard's inequality, we have%
\begin{equation*}
\int_{0}^{1}\left\vert f^{\prime }(tx+(1-t)a)\right\vert ^{q}dt\leq \frac{%
\left\vert f^{\prime }(a)\right\vert ^{q}+\left\vert f^{\prime
}(x)\right\vert ^{q}}{2}
\end{equation*}%
and%
\begin{equation*}
\int_{0}^{1}\left\vert f^{\prime }(tx+(1-t)b)\right\vert ^{q}dt\leq \frac{%
\left\vert f^{\prime }(b)\right\vert ^{q}+\left\vert f^{\prime
}(x)\right\vert ^{q}}{2},
\end{equation*}%
so%
\begin{eqnarray*}
&&\left\vert \frac{\left( b-x\right) f(b)+(x-a)f(a)}{b-a}-\frac{1}{b-a}%
\int_{a}^{b}f(u)du\right\vert \\
&\leq &\left( \frac{1}{p+1}\right) ^{\frac{1}{p}}\left( \frac{1}{2}\right) ^{%
\frac{1}{q}} \\
&&\times \left[ \frac{\left( x-a\right) ^{2}\left[ \left\vert f^{\prime
}(a)\right\vert ^{q}+\left\vert f^{\prime }(x)\right\vert ^{q}\right] ^{%
\frac{1}{q}}+\left( b-x\right) ^{2}\left[ \left\vert f^{\prime
}(x)\right\vert ^{q}+\left\vert f^{\prime }(b)\right\vert ^{q}\right] ^{%
\frac{1}{q}}}{b-a}\right]
\end{eqnarray*}%
which completes the proof.
\end{proof}

\begin{corollary}
\label{c.2.2} In Theorem \ref{t.2.2}, if we choose $x=\frac{a+b}{2}$, we
obtain 
\begin{eqnarray*}
&&\left\vert \frac{f(a)+f(b)}{2}-\frac{1}{b-a}\int_{a}^{b}f(u)du\right\vert
\\
&\leq &\frac{b-a}{4}\left( \frac{1}{p+1}\right) ^{\frac{1}{p}}\left( \frac{1%
}{2}\right) ^{\frac{1}{q}} \\
&&\times \left[ \left( \left\vert f^{\prime }(a)\right\vert ^{q}+\left\vert
f^{\prime }\left( \frac{a+b}{2}\right) \right\vert ^{q}\right) ^{\frac{1}{q}%
}+\left( \left\vert f^{\prime }(b)\right\vert ^{q}+\left\vert f^{\prime
}\left( \frac{a+b}{2}\right) \right\vert ^{q}\right) ^{\frac{1}{q}}\right] \\
&\leq &\frac{b-a}{2}\left( \frac{1}{p+1}\right) ^{\frac{1}{p}}\left( \frac{1%
}{2}\right) ^{\frac{1}{q}}\left( \left\vert f^{\prime }(a)\right\vert
+\left\vert f^{\prime }(b)\right\vert \right) .
\end{eqnarray*}%
The second inequality is obtained using the following fact: $%
\sum_{k=1}^{n}\left( a_{k}+b_{k}\right) ^{s}\leq \sum_{k=1}^{n}\left(
a_{k}\right) ^{s}+\sum_{k=1}^{n}\left( b_{k}\right) ^{s}$ for $\left( 0\leq
s<1\right) $, $a_{1},a_{2},a_{3},...,a_{n}\geq 0$ ; $%
b_{1},b_{2},b_{3},...,b_{n}\geq 0$ with $0\leq \frac{p-1}{p}<1$, for $p>1$.
\end{corollary}

\begin{theorem}
\label{t.2.3} Let $f:I\subseteq 
\mathbb{R}
\rightarrow 
\mathbb{R}
$ be a differentiable mapping on $I^{\circ }$such that $f^{\prime }\in L%
\left[ a,b\right] $, where $a,b\in I$ with $a<b$. If $\left\vert f^{\prime
}\right\vert ^{q}$ is concave on $\left[ a,b\right] $, for some fixed $q>1$,
then the following inequality holds:%
\begin{eqnarray*}
&&\left\vert \frac{\left( b-x\right) f(b)+(x-a)f(a)}{b-a}-\frac{1}{b-a}%
\int_{a}^{b}f(u)du\right\vert \\
&\leq &\left[ \frac{q-1}{2q-1}\right] ^{\frac{q-1}{q}}\left[ \frac{\left(
x-a\right) ^{2}\left\vert f^{\prime }\left( \frac{a+x}{2}\right) \right\vert
+\left( b-x\right) ^{2}\left\vert f^{\prime }\left( \frac{b+x}{2}\right)
\right\vert }{b-a}\right]
\end{eqnarray*}%
for each $x\in \left[ a,b\right] $.
\end{theorem}

\begin{proof}
As in Theorem \ref{t.2.2}, using Lemma \ref{2.1} and the well-known H\"{o}%
lder integral inequality for $q>1$ and $p=\frac{q}{q-1}$, we have 
\begin{eqnarray*}
&&\left\vert \frac{\left( b-x\right) f(b)+(x-a)f(a)}{b-a}-\frac{1}{b-a}%
\int_{a}^{b}f(u)du\right\vert \\
&\leq &\frac{\left( x-a\right) ^{2}}{b-a}\int_{0}^{1}\left( 1-t\right)
\left\vert f^{\prime }(tx+(1-t)a)\right\vert dt \\
&&+\frac{\left( b-x\right) ^{2}}{b-a}\int_{0}^{1}\left( 1-t\right)
\left\vert f^{\prime }(tx+(1-t)b)\right\vert dt \\
&\leq &\frac{\left( x-a\right) ^{2}}{b-a}\left( \int_{0}^{1}\left(
1-t\right) ^{\frac{q}{q-1}}dt\right) ^{\frac{q-1}{q}}\left(
\int_{0}^{1}\left\vert f^{\prime }(tx+(1-t)a)\right\vert ^{q}dt\right) ^{%
\frac{1}{q}} \\
&&+\frac{\left( b-x\right) ^{2}}{b-a}\left( \int_{0}^{1}\left( 1-t\right) ^{%
\frac{q}{q-1}}dt\right) ^{\frac{q-1}{q}}\left( \int_{0}^{1}\left\vert
f^{\prime }(tx+(1-t)b)\right\vert ^{q}dt\right) ^{\frac{1}{q}}.
\end{eqnarray*}%
Since $\left\vert f^{\prime }\right\vert ^{q}$ is concave on $\left[ a,b%
\right] $, we can use the Jensen's integral inequality to obtain:%
\begin{eqnarray*}
\int_{0}^{1}\left\vert f^{\prime }\left( tx+\left( 1-t\right) a\right)
\right\vert ^{q}dt &=&\int_{0}^{1}t^{0}\left\vert f^{\prime }\left(
tx+\left( 1-t\right) a\right) \right\vert ^{q}dt \\
&\leq &\left( \int_{0}^{1}t^{0}dt\right) \left\vert f^{\prime }\left( \frac{1%
}{\int_{0}^{1}t^{0}dt}\int_{0}^{1}\left( tx+\left( 1-t\right) a\right)
dt\right) \right\vert ^{q} \\
&=&\left\vert f^{\prime }\left( \frac{a+x}{2}\right) \right\vert ^{q}.
\end{eqnarray*}
Analogously,%
\begin{equation*}
\int_{0}^{1}\left\vert f^{\prime }\left( tx+\left( 1-t\right) b\right)
\right\vert ^{q}dt\leq \left\vert f^{\prime }\left( \frac{b+x}{2}\right)
\right\vert ^{q}.
\end{equation*}%
Combining all the obtained inequalities, we get 
\begin{eqnarray*}
&&\left\vert \frac{\left( b-x\right) f(b)+(x-a)f(a)}{b-a}-\frac{1}{b-a}%
\int_{a}^{b}f(u)du\right\vert \\
&\leq &\left[ \frac{q-1}{2q-1}\right] ^{\frac{q-1}{q}}\left[ \frac{\left(
x-a\right) ^{2}\left\vert f^{\prime }\left( \frac{a+x}{2}\right) \right\vert
+\left( b-x\right) ^{2}\left\vert f^{\prime }\left( \frac{b+x}{2}\right)
\right\vert }{b-a}\right]
\end{eqnarray*}%
which completes the proof.
\end{proof}

\begin{remark}
\label{r.2.2} In Theorem \ref{t.2.3}, if we choose $x=\frac{a+b}{2}$, we
have 
\begin{eqnarray*}
&&\left\vert \frac{f(a)+f(b)}{2}-\frac{1}{b-a}\int_{a}^{b}f(u)du\right\vert
\\
&\leq &\left[ \frac{q-1}{2q-1}\right] ^{\frac{q-1}{q}}\left( \frac{b-a}{4}%
\right) \left( \left\vert f^{\prime }\left( \frac{3a+b}{4}\right)
\right\vert +\left\vert f^{\prime }\left( \frac{a+3b}{4}\right) \right\vert
\right)
\end{eqnarray*}%
which is the inequality in (\ref{1.3}).
\end{remark}

\begin{theorem}
\label{t.2.4} Let $f:I\subseteq 
\mathbb{R}
\rightarrow 
\mathbb{R}
$ be a differentiable mapping on $I^{\circ }$such that $f^{\prime }\in L%
\left[ a,b\right] $, where $a,b\in I$ with $a<b$. If $\left\vert f^{\prime
}\right\vert ^{q}$ is convex on $\left[ a,b\right] $, for some fixed $q\geq
1 $, then the following inequality holds:%
\begin{eqnarray*}
&&\left\vert \frac{\left( b-x\right) f(b)+(x-a)f(a)}{b-a}-\frac{1}{b-a}%
\int_{a}^{b}f(u)du\right\vert \\
&\leq &\frac{1}{2}\left( \frac{1}{3}\right) ^{\frac{1}{q}}\left[ \frac{%
\left( x-a\right) ^{2}\left[ \left\vert f^{\prime }\left( x\right)
\right\vert ^{q}+2\left\vert f^{\prime }\left( a\right) \right\vert ^{q}%
\right] ^{\frac{1}{q}}+\left( b-x\right) ^{2}\left[ \left\vert f^{\prime
}\left( x\right) \right\vert ^{q}+2\left\vert f^{\prime }\left( b\right)
\right\vert ^{q}\right] ^{\frac{1}{q}}}{b-a}\right]
\end{eqnarray*}%
for each $x\in \left[ a,b\right] $.
\end{theorem}

\begin{proof}
Suppose that $q\geq 1$. From Lemma \ref{2.1} and using the well-known
power-mean inequality, we have 
\begin{eqnarray*}
&&\left\vert \frac{\left( b-x\right) f(b)+(x-a)f(a)}{b-a}-\frac{1}{b-a}%
\int_{a}^{b}f(u)du\right\vert  \\
&\leq &\frac{\left( x-a\right) ^{2}}{b-a}\int_{0}^{1}\left( 1-t\right)
\left\vert f^{\prime }(tx+(1-t)a)\right\vert dt \\
&&+\frac{\left( b-x\right) ^{2}}{b-a}\int_{0}^{1}\left( 1-t\right)
\left\vert f^{\prime }(tx+(1-t)b)\right\vert dt \\
&\leq &\frac{\left( x-a\right) ^{2}}{b-a}\left( \int_{0}^{1}\left(
1-t\right) dt\right) ^{1-\frac{1}{q}}\left( \int_{0}^{1}\left( 1-t\right)
\left\vert f^{\prime }(tx+(1-t)a)\right\vert ^{q}dt\right) ^{\frac{1}{q}} \\
&&+\frac{\left( b-x\right) ^{2}}{b-a}\left( \int_{0}^{1}\left( 1-t\right)
dt\right) ^{1-\frac{1}{q}}\left( \int_{0}^{1}\left( 1-t\right) \left\vert
f^{\prime }(tx+(1-t)b)\right\vert ^{q}dt\right) ^{\frac{1}{q}}.
\end{eqnarray*}%
Since $\left\vert f^{\prime }\right\vert ^{q}$ is convex, therefore we have%
\begin{eqnarray*}
&&\int_{0}^{1}\left( 1-t\right) \left\vert f^{\prime }(tx+(1-t)a)\right\vert
^{q}dt \\
&\leq &\int_{0}^{1}\left( 1-t\right) \left[ t\left\vert f^{\prime }\left(
x\right) \right\vert ^{q}+\left( 1-t\right) \left\vert f^{\prime }\left(
a\right) \right\vert ^{q}\right] dt \\
&=&\frac{\left\vert f^{\prime }\left( x\right) \right\vert ^{q}+2\left\vert
f^{\prime }\left( a\right) \right\vert ^{q}}{6}
\end{eqnarray*}%
Analogously,%
\begin{equation*}
\int_{0}^{1}\left( 1-t\right) \left\vert f^{\prime }(tx+(1-t)b)\right\vert
^{q}dt\leq \frac{\left\vert f^{\prime }\left( x\right) \right\vert
^{q}+2\left\vert f^{\prime }\left( b\right) \right\vert ^{q}}{6}.
\end{equation*}%
Combining all the above inequalities gives the desired result.
\end{proof}

\begin{corollary}
\label{c.2.3} In Theorem \ref{t.2.4}, choosing $x=\frac{a+b}{2}$ and then
using the convexity of $\left\vert f^{\prime }\right\vert ^{q}$, we have%
\begin{eqnarray*}
&&\left\vert \frac{f(a)+f(b)}{2}-\frac{1}{b-a}\int_{a}^{b}f(u)du\right\vert
\\
&\leq &\left( \frac{b-a}{8}\right) \left( \frac{1}{3}\right) ^{\frac{1}{q}}%
\left[ \left( 2\left\vert f^{\prime }(a)\right\vert ^{q}+\left\vert
f^{\prime }\left( \frac{a+b}{2}\right) \right\vert ^{q}\right) ^{\frac{1}{q}%
}+\left( 2\left\vert f^{\prime }(b)\right\vert ^{q}+\left\vert f^{\prime
}\left( \frac{a+b}{2}\right) \right\vert ^{q}\right) ^{\frac{1}{q}}\right] \\
&\leq &\left( \frac{3^{1-\frac{1}{q}}}{8}\right) \left( b-a\right) \left(
\left\vert f^{\prime }\left( a\right) \right\vert +\left\vert f^{\prime
}\left( b\right) \right\vert \right) .
\end{eqnarray*}
\end{corollary}

\begin{theorem}
\label{t.2.5} Let $f:I\subseteq 
\mathbb{R}
\rightarrow 
\mathbb{R}
$ be a differentiable mapping on $I^{\circ }$such that $f^{\prime }\in L%
\left[ a,b\right] $, where $a,b\in I$ with $a<b$. If $\left\vert f^{\prime
}\right\vert ^{q}$ is concave on $\left[ a,b\right] $, for some fixed $q\geq
1$, then the following inequality holds:%
\begin{eqnarray*}
&&\left\vert \frac{\left( b-x\right) f(b)+(x-a)f(a)}{b-a}-\frac{1}{b-a}%
\int_{a}^{b}f(u)du\right\vert \\
&\leq &\frac{1}{2}\left[ \frac{\left( x-a\right) ^{2}\left\vert f^{\prime
}\left( \frac{x+2a}{3}\right) \right\vert +\left( b-x\right) ^{2}\left\vert
f^{\prime }\left( \frac{x+2b}{3}\right) \right\vert }{b-a}\right] .
\end{eqnarray*}
\end{theorem}

\begin{proof}
First, we note that by the concavity of $\left\vert f^{\prime }\right\vert
^{q}$ and the power-mean inequality, we have%
\begin{equation*}
\left\vert f^{\prime }\left( tx+\left( 1-t\right) a\right) \right\vert
^{q}\geq t\left\vert f^{\prime }(x)\right\vert ^{q}+\left( 1-t\right)
\left\vert f^{\prime }\left( a\right) \right\vert ^{q}.
\end{equation*}%
Hence,%
\begin{equation*}
\left\vert f^{\prime }\left( tx+\left( 1-t\right) a\right) \right\vert \geq
t\left\vert f^{\prime }(x)\right\vert +\left( 1-t\right) \left\vert
f^{\prime }\left( a\right) \right\vert ,
\end{equation*}%
so $\left\vert f^{\prime }\right\vert $ is also concave.

Accordingly, using Lemma \ref{2.1} and the Jensen integral inequality, we
have 
\begin{eqnarray*}
&&\left\vert \frac{\left( b-x\right) f(b)+(x-a)f(a)}{b-a}-\frac{1}{b-a}%
\int_{a}^{b}f(u)du\right\vert \\
&\leq &\frac{\left( x-a\right) ^{2}}{b-a}\int_{0}^{1}\left( 1-t\right)
\left\vert f^{\prime }(tx+(1-t)a)\right\vert dt \\
&&+\frac{\left( b-x\right) ^{2}}{b-a}\int_{0}^{1}\left( 1-t\right)
\left\vert f^{\prime }(tx+(1-t)b)\right\vert dt \\
&\leq &\frac{\left( x-a\right) ^{2}}{b-a}\left( \int_{0}^{1}\left(
1-t\right) dt\right) \left\vert f^{\prime }\left( \frac{\int_{0}^{1}\left(
1-t\right) \left( tx+\left( 1-t\right) a\right) dt}{\int_{0}^{1}\left(
1-t\right) dt}\right) \right\vert \\
&&+\frac{\left( b-x\right) ^{2}}{b-a}\left( \int_{0}^{1}\left( 1-t\right)
dt\right) \left\vert f^{\prime }\left( \frac{\int_{0}^{1}\left( 1-t\right)
\left( tx+\left( 1-t\right) b\right) dt}{\int_{0}^{1}\left( 1-t\right) dt}%
\right) \right\vert \\
&\leq &\frac{1}{2}\left[ \frac{\left( x-a\right) ^{2}\left\vert f^{\prime
}\left( \frac{x+2a}{3}\right) \right\vert +\left( b-x\right) ^{2}\left\vert
f^{\prime }\left( \frac{x+2b}{3}\right) \right\vert }{b-a}\right] .
\end{eqnarray*}
\end{proof}

\begin{corollary}
\label{c.2.4} In Theorem \ref{t.2.5}, if we choose $x=\frac{a+b}{2}$, we have%
\begin{eqnarray*}
&&\left\vert \frac{f(a)+f(b)}{2}-\frac{1}{b-a}\int_{a}^{b}f(u)du\right\vert
\\
&\leq &\frac{b-a}{8}\left[ \left\vert f^{\prime }\left( \frac{5a+b}{6}%
\right) \right\vert +\left\vert f^{\prime }\left( \frac{a+5b}{6}\right)
\right\vert \right] .
\end{eqnarray*}
\end{corollary}

\section{APPLICATIONS TO SPECIAL MEANS}

Recall the following means which could be considered extensions of
arithmetic, logarithmic and generalized logarithmic from positive to real
numbers.

\begin{enumerate}
\item The arithmetic mean:%
\begin{equation*}
A=A\left( a,b\right) =\frac{a+b}{2};\text{ }a,b\in 
\mathbb{R}%
\end{equation*}

\item The logarithmic mean:%
\begin{equation*}
L\left( a,b\right) =\frac{b-a}{\ln \left\vert b\right\vert -\ln \left\vert
a\right\vert };\text{ }\left\vert a\right\vert \neq \left\vert b\right\vert ,%
\text{ }ab\neq 0,\text{ }a,b\in 
\mathbb{R}%
\end{equation*}

\item The generalized logarithmic mean:%
\begin{equation*}
L_{n}\left( a,b\right) =\left[ \frac{b^{n+1}-a^{n+1}}{\left( b-a\right)
\left( n+1\right) }\right] ^{\frac{1}{n}};\text{ }n\in 
\mathbb{Z}
\backslash \left\{ -1,0\right\} ,\text{ }a,b\in 
\mathbb{R}
,\text{ }a\neq b
\end{equation*}
\end{enumerate}

Now using the results of Section 2, we give some applications to special
means of real numbers.

\begin{proposition}
\label{p.3.1} Let $a,b\in 
\mathbb{R}
,$ $a<b,$ $0\notin \left[ a,b\right] $ and $n\in 
\mathbb{Z}
,$ $\left\vert n\right\vert \geq 2$. Then, for all $p>1$
\end{proposition}

(a)%
\begin{equation}
\left\vert A\left( a^{n},b^{n}\right) -L_{n}^{n}\left( a,b\right)
\right\vert \leq \left\vert n\right\vert \left( b-a\right) \left( \frac{1}{%
p+1}\right) ^{\frac{1}{p}}\left( \frac{1}{2}\right) ^{\frac{1}{q}}A\left(
\left\vert a\right\vert ^{n-1},\left\vert b\right\vert ^{n-1}\right)
\label{3.1}
\end{equation}%
and

(b)%
\begin{equation}
\left\vert A\left( a^{n},b^{n}\right) -L_{n}^{n}\left( a,b\right)
\right\vert \leq \left\vert n\right\vert \left( b-a\right) \frac{3^{1-\frac{1%
}{q}}}{4}A\left( \left\vert a\right\vert ^{n-1},\left\vert b\right\vert
^{n-1}\right) .  \label{3.2}
\end{equation}

\begin{proof}
The assertion follows from Corollary \ref{c.2.2} and \ref{c.2.3} for $%
f\left( x\right) =x^{n},$ $x\in 
\mathbb{R}
,$ $n\in 
\mathbb{Z}
,$ $\left\vert n\right\vert \geq 2$.
\end{proof}

\begin{proposition}
\label{p.3.2} Let $a,b\in 
\mathbb{R}
,$ $a<b,$ $0\notin \left[ a,b\right] $. Then, for all $q\geq 1,$
\end{proposition}

(a)%
\begin{equation}
\left\vert A\left( a^{-1},b^{-1}\right) -L^{-1}(a,b)\right\vert \leq \left(
b-a\right) \left( \frac{1}{p+1}\right) ^{\frac{1}{p}}\left( \frac{1}{2}%
\right) ^{\frac{1}{q}}A\left( \left\vert a\right\vert ^{-2},\left\vert
b\right\vert ^{-2}\right)  \label{3.3}
\end{equation}%
and

(b)%
\begin{equation}
\left\vert A\left( a^{-1},b^{-1}\right) -L^{-1}(a,b)\right\vert \leq \left(
b-a\right) \left( \frac{3^{1-\frac{1}{q}}}{4}\right) A\left( \left\vert
a\right\vert ^{-2},\left\vert b\right\vert ^{-2}\right) .  \label{3.4}
\end{equation}

\begin{proof}
The assertion follows from Corollary \ref{c.2.2} and \ref{c.2.3} for $%
f\left( x\right) =\frac{1}{x}.$
\end{proof}

\section{THE TRAPEZOIDAL FORMULA}

Let d be a division $a=x_{0}<x_{1}<...<x_{n-1}<x_{n}=b$ of the interval $%
\left[ a,b\right] $ and consider the quadrature formula 
\begin{equation}
\int_{a}^{b}f\left( x\right) dx=T\left( f,d\right) +E\left( f,d\right)
\label{4.1}
\end{equation}%
where%
\begin{equation*}
T\left( f,d\right) =\sum_{i=0}^{n-1}\frac{f\left( x_{i}\right) +f\left(
x_{i+1}\right) }{2}\left( x_{i+1}-x_{i}\right)
\end{equation*}%
for the trapezoidal version and $E\left( f,d\right) $ denotes the associated
approximation error.

\begin{proposition}
\label{p.4.1} Let $f:I\subseteq 
\mathbb{R}
\rightarrow 
\mathbb{R}
$ be a differentiable mapping on $I^{\circ }$such that $f^{\prime }\in L%
\left[ a,b\right] $, where $a,b\in I$ with $a<b$ and $\left\vert f^{\prime
}\right\vert ^{\frac{p}{p-1}}$ is convex on $\left[ a,b\right] $, where $p>1$%
. Then in $\left( \ref{4.1}\right) $, for every division $d$ of $\left[ a,b%
\right] ,$ the trapezoidal error estimate satisfies%
\begin{equation*}
\left\vert E\left( f,d\right) \right\vert \leq \left( \frac{1}{p+1}\right) ^{%
\frac{1}{p}}\left( \frac{1}{2}\right) ^{\frac{1}{q}}\sum_{i=0}^{n-1}\frac{%
\left( x_{i+1}-x_{i}\right) ^{2}}{2}\left( \left\vert f^{\prime }\left(
x_{i}\right) \right\vert +\left\vert f^{\prime }\left( x_{i+1}\right)
\right\vert \right) .
\end{equation*}
\end{proposition}

\begin{proof}
On applying Corollary \ref{c.2.2} on the subinterval $\left[ x_{i},x_{i+1}%
\right] \left( i=0,1,2,...,n-1\right) $ of the division, we have%
\begin{eqnarray*}
&&\left\vert \frac{f\left( x_{i}\right) +f\left( x_{i+1}\right) }{2}-\frac{1%
}{x_{i+1}-x_{i}}\int_{x_{i}}^{x_{i+1}}f\left( x\right) dx\right\vert \\
&\leq &\frac{\left( x_{i+1}-x_{i}\right) }{2}\left( \frac{1}{p+1}\right) ^{%
\frac{1}{p}}\left( \frac{1}{2}\right) ^{\frac{1}{q}}\left( \left\vert
f^{\prime }\left( x_{i}\right) \right\vert +\left\vert f^{\prime }\left(
x_{i+1}\right) \right\vert \right) .
\end{eqnarray*}%
Hence in $\left( \ref{4.1}\right) $ we have%
\begin{eqnarray*}
\left\vert \int_{a}^{b}f(x)dx-T\left( f,d\right) \right\vert &=&\left\vert
\sum_{i=0}^{n-1}\left\{ \int_{x_{i}}^{x_{i+1}}f(x)dx-\frac{f\left(
x_{i}\right) +f\left( x_{i+1}\right) }{2}\left( x_{i+1}-x_{i}\right)
\right\} \right\vert \\
&\leq &\sum_{i=0}^{n-1}\left\vert \int_{x_{i}}^{x_{i+1}}f(x)dx-\frac{f\left(
x_{i}\right) +f\left( x_{i+1}\right) }{2}\left( x_{i+1}-x_{i}\right)
\right\vert \\
&\leq &\left( \frac{1}{p+1}\right) ^{\frac{1}{p}}\left( \frac{1}{2}\right) ^{%
\frac{1}{q}}\sum_{i=0}^{n-1}\frac{\left( x_{i+1}-x_{i}\right) ^{2}}{2}\left(
\left\vert f^{\prime }\left( x_{i}\right) \right\vert +\left\vert f^{\prime
}\left( x_{i+1}\right) \right\vert \right)
\end{eqnarray*}%
which completes the proof.
\end{proof}

\begin{proposition}
\label{p.4.2} Let $f:I\subseteq 
\mathbb{R}
\rightarrow 
\mathbb{R}
$ be a differentiable mapping on $I^{\circ }$such that $f^{\prime }\in L%
\left[ a,b\right] $, where $a,b\in I$ with $a<b$. If $\left\vert f^{\prime
}\right\vert ^{q}$ is concave on $\left[ a,b\right] $, for some fixed $q>1$.
Then in $\left( \ref{4.1}\right) $, for every division $d$ of $\left[ a,b%
\right] ,$ the trapezoidal error estimate satisfies%
\begin{equation*}
\left\vert E\left( f,d\right) \right\vert \leq \left( \frac{q-1}{2q-1}%
\right) ^{\frac{q-1}{q}}\sum_{i=0}^{n-1}\frac{\left( x_{i+1}-x_{i}\right)
^{2}}{4}\left( \left\vert f^{\prime }\left( \frac{3x_{i}+x_{i+1}}{4}\right)
\right\vert +\left\vert f^{\prime }\left( \frac{x_{i}+3x_{i+1}}{4}\right)
\right\vert \right) .
\end{equation*}
\end{proposition}

\begin{proof}
The proof is similar to that of Proposition \ref{p.4.1} and using Remark \ref%
{r.2.2}.
\end{proof}

\begin{proposition}
\label{p.4.3} Let $f:I\subseteq 
\mathbb{R}
\rightarrow 
\mathbb{R}
$ be a differentiable mapping on $I^{\circ }$such that $f^{\prime }\in L%
\left[ a,b\right] $, where $a,b\in I$ with $a<b$. If $\left\vert f^{\prime
}\right\vert ^{q}$ is concave on $\left[ a,b\right] $, for some fixed $q\geq
1$.Then in $\left( \ref{4.1}\right) $, for every division $d$ of $\left[ a,b%
\right] ,$ the trapezoidal error estimate satisfies%
\begin{equation*}
\left\vert E\left( f,d\right) \right\vert \leq \frac{1}{8}%
\sum_{i=0}^{n-1}\left( x_{i+1}-x_{i}\right) ^{2}\left( \left\vert f^{\prime
}\left( \frac{5x_{i}+x_{i+1}}{6}\right) \right\vert +\left\vert f^{\prime
}\left( \frac{x_{i}+5x_{i+1}}{6}\right) \right\vert \right) .
\end{equation*}
\end{proposition}

\begin{proof}
The proof is similar to that of Proposition \ref{p.4.1} and using Corollary %
\ref{c.2.4}.
\end{proof}


\begin{thebibliography}{9}
\bibitem{CPJP} C.E.M Pearce, J. Pe\v{c}ari\'{c}, Inequalities for
differentiable mappings with application to special means and quadrature
formula, Appl. Math. Lett., 13 (2000) 51-55.

\bibitem{PPT} J.E. Pe\v{c}ari\'{c}, F. Proschan, Y.L. Tong, Convex
Functions, Partial Ordering and Statistical Applications, Academic Press,
New York, 1991.

\bibitem{SDRA} S.S. Dragomir, R.P. Agarwal, Two inequalities for
differentiable mappings and applications to special means of real numbers
and to trapezoidal formula, Appl. Math. Lett., 11 (5) (1998) 91-95.

\bibitem{UMMJ} U.S. K\i rmac\i , M. Klari\v{c}i\'{c} Bakula, M.E. \"{O}%
zdemir, J. Pe\v{c}ari\'{c}, Hadamard-type inequalities for $s-$convex
functions, Appl. Math. Comput., 193 (2007) 26-35.
\end{thebibliography}
\end{document}